\documentclass{amsart}
\usepackage{amssymb}
\usepackage{color}
\usepackage{hyperref}

\usepackage{tikz-cd}
\usepackage{rotating}

\newtheorem{theorem}{Theorem}[section]
\newtheorem{lemma}[theorem]{Lemma}
\newtheorem{proposition}[theorem]{Proposition}
\newtheorem{corollary}[theorem]{Corollary}
\theoremstyle{definition}
\newtheorem{definition}[theorem]{Definition}
\newtheorem{example}[theorem]{Example}

\newtheorem{remark}[theorem]{Remark}
\numberwithin{equation}{section}
\usepackage[all]{xy}
\DeclareMathOperator{\Po}{Po}
\DeclareMathOperator{\Ost}{Ost}
\DeclareMathOperator{\Cl}{Cl}
\DeclareMathOperator{\disc}{disc}

\DeclareMathOperator{\Ker}{Ker}
\DeclareMathOperator{\Coker}{Coker}
\DeclareMathOperator{\Inf}{Inf}

\usepackage[all]{xy}

\begin{document}

\title{Ostrowski quotients for finite extensions of number fields}

\author{Ehsan Shahoseini}
\address{Department of Mathematics, Tarbiat Modares University, 14115-134, Tehran, Iran}
\curraddr{} \email{ehsan\_shahoseini@modares.ac.ir}
\thanks{}

\author{Ali Rajaei$^{*}$}
\address{Department of Mathematics, Tarbiat Modares University, 14115-134, Tehran, Iran}
\curraddr{}
\email{alirajaei@modares.ac.ir}
\thanks{$^{*}$Corresponding author}

\author{Abbas Maarefparvar$^{1}$}
\address{School of Mathematics, Institute for Research in Fundamental Sciences (IPM), P.O. Box: 19395-5746, Tehran, Iran.}
\curraddr{}
\email{a.marefparvar@ipm.ir}
\thanks{$^{1}$The third author's research was supported by a grant from IPM}
\date{}
\dedicatory{}
\commby{}


\begin{abstract}
	For $L/K$ a finite Galois extension of number fields, the  relative P\'olya group $\Po(L/K)$ coincides with the group of strongly ambiguous ideal classes in $L/K$. In this paper, using a well known exact sequence related to $\Po(L/K)$, in the works of Brumer-Rosen and Zantema, we find short proofs for some classical results in the literatur. Then we define the ``Ostrowski quotient'' $\Ost(L/K)$ as the cokernel of the capitulation map into $\Po(L/K)$, and generalize some known results for $\Po(L/\mathbb{Q})$ to $\Ost(L/K)$.
\end{abstract}

\maketitle
 
\vspace{.2cm} {\noindent \bf{Keywords:}}~ P\'olya group, relative P\'olya group,  Ostrowski quotient, Galois cohomology, capitulation problem.

\vspace{.2cm} {\noindent \bf{Notations.}}~  The following notations will be used throughout this article:

For a number field $K$, the notations $I(K)$, $P(K)$, $\Cl(K)$,  $\mathcal{O}_K$, $h_K$, $U_K$, $H(K)$, and $\Gamma(K)$ denote the group of fractional ideals, group of principal fractional ideals, ideal class group, ring of integers, class number, group of units, Hilbert class field, and genus field of $K$, respectively. 

For a finite extension $L/K$ of number fields, 
$\mathcal{N}_{L/K}:\Cl(L) \rightarrow \Cl(K)$ denotes the induced  homomorphism by the ideal norm homomorphism $N_{L/K}:I(L) \rightarrow I(K)$; and  ${\epsilon}_{L/K}:\Cl(K) \rightarrow \Cl(L)$ denotes the transfer of ideal classes induced by the homomorphism $j_{L/K}: \mathfrak{a} \in I(K)  \mapsto \mathfrak{a} \mathcal{O}_L \in  I(L)$.
Whenever $L/K$ is Galois, for a prime ideal $\mathfrak{P}$  of $K$ denote by $e_{\mathfrak{P}(L/K)}$ and  $f_{\mathfrak{P}(L/K)}$  the ramification index and residue class degree of $\mathfrak{P}$ in $L/K$, respectively.

 

\section{Introduction}
Let $L$ be an algebraic number field with ring of integers $\mathcal{O}_L$. For every prime number $p$ and every integer $f \geq 1$, the \textit{Ostrowski ideal} $\Pi_{p^f}(L)$ of $L$ is defined as follows \cite{Ostrowski}:
\begin{equation} \label{equation, Ostrowski ideals}
\Pi_{p^f}(L) :=\prod_{\substack{\mathfrak{m}\in Max(\mathcal{O}_L)\\ \mathcal{N}_{L/\mathbb{Q}}(\mathfrak{m})=p^f}} \mathfrak{m},
\end{equation}
 where by the convention, if $L$ has no ideal with norm $p^f$, then $\Pi_{p^f}(L)=\mathcal{O}_L$.
Following Zantema \cite{Zantema}, $L$ is called a \textit{P\'olya field} if all the Ostrowski ideals $\Pi_{p^f}(L)$ of $L$ for arbitrary prime powers $p^f$ are principal. As an ``obstruction measure'' for $L$ to be a P\'olya field, the notion of \textit{P\'olya-Ostrowski group} or \textit{P\'olya group} was introduced in \cite{Cahen-Chabert's book}:
\begin{definition} \label{definition, Polya group}
The P\'olya group of a number field $L$ is the subgroup $\Po(L)$ of
the class group $\Cl(L)$ generated by the classes of the Ostrowski ideals $\Pi_q(L)$.
\end{definition} 

Hence $L$ is P\'olya iff its P\'olya group is trivial. Obviously every number field with class number one is a P\'olya field, but not conversely, see e.g. Proposition \ref{proposition, Hilbert's result for Polya quadratic field} below. 

Now let $L$ be a Galois extension of $\mathbb{Q}$ with Galois group $G$. In this case, P\'olyaness of $L$ is equivalent to principality of $\prod_{i=1}^g \mathfrak{P}_i$, where $\mathfrak{P}_1,\mathfrak{P}_2,\dots,\mathfrak{P}_g$ are all distinct prime ideals of $L$ above a \textit{ramified prime} $p$, see  \cite[Section 1]{Zantema}. On the other hand, one can show that the \textit{Ostrowski  ideals} $\Pi_q(L)$ freely  generate the ambiguous ideals $I(L)^G$, see \cite[Section 2]{Brumer-Rosen}. Thus $\Po(L)$ is the subgroup of $\Cl(L)$ generated by the classes of ambiguous ideals. As the first result in this subject, P\'olya groups of quadratic fields have been described by Hilbert \cite{Hilbert's book}:

\begin{proposition} \cite[Theorem 105-106]{Hilbert's book} \label{proposition, Hilbert's result for Polya quadratic field}
Let $L$ be a quadratic field and denote the number of ramified primes in $L/\mathbb{Q}$ by $r_{L}$. Then:
\begin{equation*}
\#Po(L) =
\begin{cases}
2^{r_L-2} \,  &: \,   \text{$L$ is real with no unit of negative norm} \\
2^{r_L-1} \,  &: \,   \text{Otherwise}
\end{cases}
\end{equation*}

\end{proposition}

The above theorem of Hilbert has been generalized by Zantema:

\begin{proposition}  \cite[ Section 3, page 9]{Zantema} \label{proposition, Zantema's exact sequence}
Let $L/\mathbb{Q}$ be a Galois extension with Galois group $G$. Denote the ramification index of a prime $p$ in $L$ by $e_p$ .Then the following sequence is exact:
\begin{equation} \label{equation, Zantema's exact sequence}
\{0\} \longrightarrow H^1(G,U_L) \longrightarrow \bigoplus _{\text{p prime}} \mathbb{Z} / e_p \mathbb{Z} \longrightarrow \Po(L) \longrightarrow \{0\}.
\end{equation}
\end{proposition}

\begin{remark} \label{remark, introducing BRZ notion}
In \cite{MR2} we called \eqref{equation, Zantema's exact sequence} ``Zantema's exact sequence'', however a version of this sequence can be seen in \cite{Brumer-Rosen} and therefore we call a generalization of this exact sequence the ``Brumer-Rosen-Zantema'' exact sequence or shortly ``BRZ'', see Theorem \ref{theorem, generalization of the Zantema's exact sequence} below.
\end{remark}

\section{Relative P\'olya group}
Recently, the notion of P\'olya group has been generalized to \textit{Relative P\'olya group} (independently in \cite{ChabertI, MaarefparvarThesis}) in the following sense:

\begin{definition} \label{definition, relative Polya group} 
Let $L/K$ be a finite extension of number fields. The \textit{relative P\'olya group} of $L$ over $K$, is the subgroup of $\Cl(L)$ generated by the classes of the \textit{relative Ostrowski ideals}:   
\begin{align*}
\Pi_{\mathfrak{P}^f}(L/K):=\prod_{\substack{\mathfrak{M}\in Max(\mathcal{O}_L)\\ N_{L/ K}(\mathfrak{M})=\mathfrak{P}^f}} \mathfrak{M},
\end{align*}
 where $\mathfrak{P}$ is a prime ideal of $K$, $f$ is a positive integer and by the convention, if $L$ has no ideal with relative norm $\mathfrak{P}^f$ (over $K$), then $\Pi_{\mathfrak{P}^f}(L/K)=\mathcal{O}_L$.
We denote the relative P\'olya group of $L$ over $K$ by $\Po(L/K)$. In particular, $\Po(L/\mathbb{Q})=\Po(L)$ and $\Po(L/L)=\Cl(L)$.
\end{definition}

\begin{theorem} \cite{ChabertI,MR2} \label{theorem, generalization of the Zantema's exact sequence}
Let $L/K$ be a finite Galois extension of number fields with Galois group $G$. Then the following sequence is exact:
{\small
\begin{equation*} \label{equation, main exact sequence}
	\tag{BRZ} \qquad 
\{0 \} \rightarrow \Ker({\epsilon}_{L/K}) \rightarrow H^1(G,U_L) \rightarrow\bigoplus_{\mathfrak{P}} \frac{\mathbb{Z}}{e_{\mathfrak{P}(L/K)}\mathbb{Z}} \rightarrow \frac{\Po(L/K)}{{\epsilon}_{L/K}(\Cl(K))} \rightarrow \{0 \}.
\end{equation*}}

\end{theorem}

We summarize some  interesting consequences of Theorem \ref{theorem, generalization of the Zantema's exact sequence} (for more results see \cite{ChabertI,MR2}):

\begin{corollary}  \label{corollary, some consequence of the main exact sequence for RPG}
For $L/K$ a finite Galois extension of number fields with Galois group $G$, the following assertions hold:

\begin{itemize}
\item[(i)]  
\cite[Proposition 4.4]{ChabertI} $\#\Po(L/K)=\frac{h_K. \prod_{\mathfrak{P}} e_{\mathfrak{P}(L/K)}}{\#H^1(G,U_L)}$.


\item[(ii)]  \cite[Remark 2.3]{MR2}
If $\gcd(h_L,[L:K])=1$, then $\Po(L/K)={\epsilon}_{L/K}(\Cl(K))$. 

\item[(iii)] \cite[Corollary 2.4]{MR2}
 If $gcd(h_K,[L:K])=1$ then $\Cl(K) \hookrightarrow \Po(L/K)$. In particular, $h_K \mid h_L$. Moreover, the following sequence is exact:
\begin{equation*}
\{0\} \rightarrow H^1(G,U_L)  \rightarrow \bigoplus_{\mathfrak{P} | \disc(L/K)} \mathbb{Z}/e_{\mathfrak{P}} \mathbb{Z} \rightarrow \Po(L/K)/\Cl(K) \rightarrow \{ 0\}.
\end{equation*}

\item[(iv)]  \cite[Corollary 2.9]{MR2}
If every ideal class of $K$ extended to $L$ is principal and all finite places of $K$ are unramified in $L$, then $\Cl(K) \simeq H^1(G, U_L)$ and $\Po(L/K) = \{0\}$. In particular, $\Po(H(K)/K)$ is trivial, for the Hilbert class field $H(K)$ of $K$.
\end{itemize}
\end{corollary}

\begin{remark}
Indeed classically, $\#Po(L/K)$ can be obtained also from Brumer-Rosen results in \cite[Lemma 2.1 and Proposition 2.2]{Brumer-Rosen}.
\end{remark}

\begin{lemma} \cite[Lemma 2.10]{MR2} \label{lemma, for P sub N sub M, if M/N is Galois are Galois, then Po(M/P) is contained in Po(M/N)}
Let $F \subseteq K \subseteq L$ be a tower of finite extensions of number fields. 

\begin{itemize}
\item[$(i)$]
If $K/F$ and $L/F$ are Galois extensions, then $\epsilon_{L/K}(\Po(K/F)) \subseteq \Po(L/F)$.
\item[$(ii)$]
If $L/K$ is a Galois extension, then $\Po(L/F) \subseteq \Po(L/K)$.
\end{itemize}

\end{lemma}

\begin{remark}
Note that if either $K/F$ or $L/F$ is not Galois, then the containment in the part $(i)$ in Lemma \ref{lemma, for P sub N sub M, if M/N is Galois are Galois, then Po(M/P) is contained in Po(M/N)}  might not hold. For instance, consider the pure cubic field $K=\mathbb{Q}(\sqrt[3]{19})$. One can show that the Galois closure $L$ of $K$ is a P\'olya field while $\Po(K)=\Cl(K)\simeq \mathbb{Z}/3 \mathbb{Z}$, see \cite[Example 2.9]{MR1}. On the other hand, since $gcd(h_K,[L:K])=1$, $\mathcal{N}_{L/K} \, \circ \, {\epsilon}_{L/K}:\bar{\mathfrak{a}} \in \Cl(K) \mapsto \bar{\mathfrak{a}}^{[L:K]} \in \Cl(K)$ is injective, hence so is ${\epsilon}_{L/K}$. Therefore ${\epsilon}_{L/K}(\Po(K))=\Po(K)=\Cl(K) \not \subseteq \Po(L)$. 
\end{remark}

\subsection{Some Applications}
In this section, using the exact sequence \eqref{equation, main exact sequence} we find easier and shorter  proofs for some results in the literature. 
We begin  with the following result of Tannaka:
\begin{proposition}  \cite[Theorem 8]{Tannaka} \label{proposition, Tannaka result}
For a number field $K$, one has
\begin{equation} \label{equation, isomorphism between H^1 and Galios group for the Hilbert class field}
H^1(Gal(H(K)/K),U_{H(K)}) \simeq Gal(H(K)/K),
\end{equation}
where $H(K)$ denotes the Hilbert class field  of $K$.
\end{proposition}
\begin{proof}
By the principal ideal theorem, $\Ker(\epsilon_{H(K)/K})=\Cl(K)$. Since $H(K)/K$ is an unramified extension,  
using the isomorphism induced by Artin reciprocity, 
and the exact sequence \eqref{equation, main exact sequence}, we find
\begin{equation*}
Gal(H(K)/K) \simeq \Cl(K)=\Ker(\epsilon_{H(K)/K}) \simeq H^1(Gal(H(K)/K),U_{H(K)}).
\end{equation*}
\end{proof}

\begin{remark}
Isomorphism \eqref{equation, isomorphism between H^1 and Galios group for the Hilbert class field} also has been  obtained  independently by Khare and Prasad, see \cite[$\S$4, Corollary 3]{Khare}. They mention that this isomorphism is equivalent to the principal ideal theorem, see \cite[$\S$4, Remark 3]{Khare}.
\end{remark}

A similar statement holds for \textit{any} finite extension of the Hilbert class field, due to Iwasawa:

\begin{proposition}\cite{Iwasawa} \label{proposition, Iwasawa result for finite extension of Hilbert class field}
For $L/K$ a  finite Galois extension of number fields which is unramified at all finite places, if $H(K) \subseteq L$ then
\begin{equation*}
H^1(Gal(L/K),U_L) \simeq Cl(K).
\end{equation*}
\end{proposition}

\begin{proof}
Since $\epsilon_{L/K}=\epsilon_{L/H(K)} \, \circ \,  \epsilon_{H(K)/K}$, by the principal ideal theorem we have $\Ker(\epsilon_{L/K})=\Cl(K)$. The exact sequence \eqref{equation, main exact sequence} yields the claim.
\end{proof}

In a more general case (without considering the Hilbert class field), Iwasawa \cite{Iwasawa} and Khare-Prasad \cite[$\S$ 4, Proposition 2]{Khare} independently proved that:
\begin{proposition}  \label{proposition, Iwasawa-Khare-Prasad result}
Let $L/K$ be a  finite Galois extension of number fields that is unramified at all finite places. Then $\Ker(\epsilon_{L/K}) \simeq H^1(Gal(L/K),U_L)$.
\end{proposition}

\begin{proof}
Immediately follows from the exact sequence \eqref{equation, main exact sequence}.
\end{proof}

\begin{remark}
Note that in our proofs, we have assumed a weaker assumption, i.e. \textit{finite places} are unramified 
versus \textit{all places}. 
For instance, for $L=\mathbb{Q}(\sqrt{-1},\sqrt{3})$ and $K=\mathbb{Q}(\sqrt{3})$, ramification happens only for infinite places of $K$ in $L$; whereas by  the exact sequence \eqref{equation, main exact sequence} $H^1(Gal(L/K),U_L)$ is trivial, since $L/K$ is unramified in all finite places (Note that $h_K=1$).
\end{remark}

As another application  of the exact sequence \eqref{equation, main exact sequence} we find the following results on the ``kernel'' and ``cokernel'' of the capitulation map in cyclic unramified extensions:

\begin{theorem} \label{theorem, generalization of Hilbert's theorem 94}
Let $L/K$ be a finite cyclic extension of number fields with Galois group $G$.
\begin{itemize}
	\item[(i)] If $L$ is unramified at all finite and infinite places of $K$, then 
	\begin{equation}  \label{equation, order of Ker(epsilon) for Hilbert's theorem 94}
	\#\Ker(\epsilon_{L/K})=\left(U_K: Norm_{L/K}(U_L) \right). [L:K].
	\end{equation}
	 
	\item[(ii)]  If  $L$ is unramified at all finite places of $K$, then
	\begin{equation} \label{equation, order of the kocernel of capitualtion map}
H^2(G,U_L) \hookrightarrow \Coker(\epsilon_{L/K}).
	\end{equation}
\end{itemize}

 \end{theorem}

\begin{proof}
$(i)$ 
Since $L/K$ is unramified, by the exact sequence \eqref{equation, main exact sequence} we have

\begin{equation} \label{equation, order of H^1 equals to order of kernel(epsiilon)}
\#{H}^1(G,U_L)=\#Ker(\epsilon_{L/K}).
\end{equation}

On the one hand, since $L/K$ is cyclic, we can use the \textit{Herbrand quotient}:
\begin{align} \label{equation, Herbrand quotients of L over K}
Q(G,U_L)=\frac{\# \hat{H}^0(G, U_L)}{\#{H}^1(G,U_L)},
\end{align}
where 
\begin{align} \label{equation, order of H^0 hat}
\hat{H}^0(G, U_L)&=\frac{U_L^{G}}{Norm_{L/K}(U_L)}=\frac{U_K}{Norm_{L/K}(U_L)}.
\end{align}
On the other hand,
\begin{align*}
Q(G,U_L)= \frac{2^s}{[L:K]},
\end{align*}
where $s$ is the number of infinite places of $K$ ramified in $L$ \cite[Chapter IX]{Lang}. By the assumption $L/K$ is unramified which implies that $s=0$ i.e.
\begin{equation} \label{equation, Herbarand quotient for cyclic unrmaified extensions}
Q(G,U_L)=\frac{1}{[L:K]}.
\end{equation}
Using relations \eqref{equation, order of H^1 equals to order of kernel(epsiilon)}--\eqref{equation, Herbarand quotient for cyclic unrmaified extensions} we obtain the desired equality.

$(ii)$ Consider the following exact sequence
\begin{equation*}
	\{0 \} \longrightarrow P(L) \longrightarrow I(L) \longrightarrow \Cl(L) \longrightarrow \{0 \}.
\end{equation*}

Taking Cohomology, we get the  sequence 
\begin{equation*}
	\{0 \} \longrightarrow P(L)^G \longrightarrow I(L)^G \longrightarrow \left(\Cl(L)\right)^G \longrightarrow H^1(G,P(L)) \longrightarrow   \{0 \}
\end{equation*}
	is exact, since $H^1(G,I(L))=\{0\}$, see e.g.  \cite[proof of Theorem 1]{Kisilevsky} or \cite[Lemma 2.1]{Van der Waal}.
	Since $L/K$ is Galois, $\Po(L/K)=\frac{I(L)^G}{P(L)^G}$, see \cite[$\S$ 2]{MR2}. Hence the last exact sequence can be rewritten as
	\begin{equation*}
		\{0 \} \longrightarrow \Po(L/K) \longrightarrow  \left(\Cl(L)\right)^G \longrightarrow H^1(G,P(L)) \longrightarrow   \{0 \},
	\end{equation*}
	which implies that
	\begin{equation} \label{equation, G-invariant ideal class group modulo relative Polya}
\frac{\left(\Cl(L)\right)^G }{\Po(L/K)} \simeq H^1(G,P(L)). 
	\end{equation}
On the other hand, Kisilevsky \cite[Lemma 1]{Kisilevsky} proved that
\begin{equation} \label{equation, Kisilevsky result}
 H^1(G,P(L)) \simeq H^2(G,U_L).
\end{equation}	
Therefore
 \begin{equation} \label{equation, G-invariant ideal class group modulo relative Polya isomorphic to H^2}
\frac{\left(\Cl(L)\right)^G }{\Po(L/K)} \simeq  H^2(G,U_L).
 \end{equation}	
Since all finite places of $K$ are unramified in $L$, the exact sequence \eqref{equation, main exact sequence} yields
\begin{equation} \label{equation, relative Polya is image epsilon in unramified case}
\Po(L/K)=\epsilon_{L/K}\left(\Cl(K)\right).
\end{equation}
Using equations \eqref{equation, G-invariant ideal class group modulo relative Polya isomorphic to H^2} and \eqref{equation, relative Polya is image epsilon in unramified case} we have
\begin{equation}
 H^2(G,U_L)\simeq \frac{\left(\Cl(L)\right)^G }{\epsilon_{L/K}\left(\Cl(K)\right)}
\end{equation}
which is a subgroup of $\Coker(\epsilon_{L/K})$.
\end{proof}

\begin{remark}
Part $(i)$ of Theorem \ref{theorem, generalization of Hilbert's theorem 94} can be  thought of as a \textit{quantitative version} of ``Hilbert's Theorem 94''  \cite{Hilbert's book}.
\end{remark}

\section{Ostrowski quotient} \label{section, Ostrowski quotient}

By exact sequence \eqref{equation, Zantema's exact sequence}, there exists a surjective map from $ \bigoplus_{p} \dfrac{\mathbb{Z}}{e_{p(L/\mathbb{Q})}\mathbb{Z}} $ onto $\Po(L)$. Whereas switching to the relative case, namely $\Po(L/K)$ for $L/K$ a finite Galois extension of number fields, this statement \textit{does not} hold, in general. There are finite Galois extensions $L/K$ such that the only map from  $ \bigoplus_{\mathfrak{p}} \dfrac{\mathbb{Z}}{e_{\mathfrak{p}(L/K)}\mathbb{Z}} $ to $ \Po(L/K) $ is the zero map while $\Po(L/K)$ is non trivial:

\begin{example} \label{example, relative Polya group for cyclotomic fields}
	Let $ K = \mathbb{Q}(\zeta_p) $ and $ L = \mathbb{Q}(\zeta_{p^n}) $ where $ n \geq 1 $ and $ p $ is a regular prime, i.e. $ p \nmid h_K $, and $\zeta_p$ (resp. $\zeta_{p^n}$) denote the $p$-th (resp. $p^n$-th) primitive root of unity. By a  result of Iwasawa \cite{IwasawaII}, since $h_K$ is not divisible by $p$, so is not $h_L$. In particular $ p \nmid \#\Po(L/K)$. The only prime of $K$ that ramifies in $L$  is $ \mathfrak{p} =< 1-\zeta_p> $  whose ramification index is 
	$[L : K] = p^{n-1} $, i.e. $ \mathfrak{p}$ is totally ramified. Therefore the only map from  $ \bigoplus_{\mathfrak{p}} \dfrac{\mathbb{Z}}{e_{\mathfrak{p}(L/K)}\mathbb{Z}} $ to $ \Po(L/K) $ is the zero map.  Masley and Montgomery \cite{Masley} proved that $h_K \neq 1$ for $p > 19$. Since $gcd(h_K,[L:K])=1$, one has $\Cl(K) \hookrightarrow \Po(L/K)$, see \cite[Corollary 2.4]{MR2}.  
	 Hence for $p > 19$, $\Po(L/K)$ is nontrivial.
	 Note that 
	 \begin{equation*} 
\mathfrak{p} \mathcal{O}_L=\left(1-\zeta_p\right)\mathcal{O}_L=\left<1-\zeta_{p^n}\right>^{p^{n-1}}
	 \end{equation*}
	  is a principal ideal of $L$, hence is trivial in $\Po(L/K)$.
\end{example}
The above example shows that the controllability of the P\'olya group by ramification, cannot transfer to the relative P\'olya group. This motivates us to  modify the notion of relative P\'olya group to arrive at a notion directly governed by ramification.

\begin{definition} \cite{Ehsan Thesis} \label{definition, Ostrowski group}
	For a finite extension $ L/K $ of number fields, the \textit{Ostrowski quotient} $\Ost(L/K)$   is defined as  
	\begin{equation} \label{equation, Ostrowski group general case}
	\Ost(L/K) := \dfrac{\Po(L/K)}{\Po(L/K) \cap \epsilon_{L/K}(Cl(K))}.
	\end{equation} 
	In particular, $\Ost(L/\mathbb{Q})=\Po(L/\mathbb{Q})=\Po(L)$ and $\Ost(L/L)=\{0\}$. The extension $L/K$ is called ``\textit{Ostrowski}''  (or $L$ is called $K$-Ostrowski) if $\Ost(L/K)$ is trivial.
\end{definition}

\begin{remark} \label{remark, if L/K is Galois then epsilon of Cl(K) is contained in relative Polya group}
	If $L/K$ is a Galois extension with Galois group $G$, then  $\epsilon_{L/K}(\Cl(K)) \subseteq \Po(L/K)$ \cite[$\S$ 2]{MR2}, and the exact sequence \eqref{equation, main exact sequence} can be rewritten as follows
\begin{equation} \label{equation, main exact sequence for the Ostrowski quotient}
{\small
\xymatrix{
\{0\} \ar[r]  &  \Ker({\epsilon}_{L/K})  \ar[r]^{\theta_{L/K}}   & H^1(G,U_L) \ar[r] & \bigoplus_{\mathfrak{P}} \frac{\mathbb{Z}}{e_{\mathfrak{P}(L/K)}\mathbb{Z}} \ar[r] & \Ost(L/K)  \ar[r] & \{0\}.
 }}
\end{equation}
 Hence in this case we have
	\begin{equation} \label{equation, order of the Ostrowski quotient}
	\#\Ost(L/K)=\frac{\#\Ker(\epsilon_{L/K}). \prod_{\mathfrak{P} | disc(L/K)} e_{\mathfrak{P}(L/K)}}{\#H^1(G,U_L)},
	\end{equation}
	which relates the Ostrowski quotient of $L/K$ to the ``\textit{capitulation kernel}''.

\end{remark}

Using Corollary \ref{corollary, some consequence of the main exact sequence for RPG}, we immediately find:

\begin{theorem}
	Let $L/K$ be a finite Galois extension of number fields with Galois group $G$.
	\begin{itemize}
		\item[(i)] If $gcd(h_K,[L:K])=1$, then the following sequence is exact:
		\begin{equation*} 
			\{0 \}  \rightarrow H^1(G,U_L) \rightarrow \bigoplus _{\mathfrak{P}} \frac{\mathbb{Z}}{e_{\mathfrak{P} (L/K)} \mathbb{Z}} \rightarrow \Ost(L/K) \rightarrow \{0 \}.
		\end{equation*}
		\item[(ii)] If either $gcd(h_L,[L:K])=1$ or all finite places of $K$ are unramified in $L$, then $L/K$ is Ostrowski. In particular, the extensions $H(K)/K$ and $\Gamma(K)/K$ are Ostrowski, where $H(K)$ and $\Gamma(K)$ denote the Hilbert class field and genus field of $K$, respectively.
	\end{itemize}
\end{theorem}

For ``abelian'' number fields, Zantema proved:

\begin{proposition} \cite[Proposition 2.5]{Zantema} \label{proposition, Zantema result about abelian number fields}
If $K/\mathbb{Q}$ is an abelian extension ramifying at only one prime, then $K$ is a P\'olya field.
\end{proposition}

 Using some notions in class field theory for finite abelian extensions, namely the ``conductor'' and the ``ray class field'' \cite[Chapter X, $\S$3]{Lang}, we can find an analogous  statement to Proposition \ref{proposition, Zantema result about abelian number fields} for the Ostrowski quotient:

\begin{theorem} \label{theorem, generalizing Zantema result for abelian number fields ramified only at one prime}
Let $K/F$ be a finite abelian extension of number fields such that only one prime of $F$ is ramified in $K$. Let $L$ be the ray class field of $F$ for the modulus $\mathfrak{c}(K/F)$, where $\mathfrak{c}(K/F)$ denotes the conductor of $K$ over $F$. If $L/F$ is Ostrowski, then so is
 $K/F$.
\end{theorem}

\begin{proof}
Let $\mathfrak{p}$ be the only prime of $F$ ramified in $K$ and  
\begin{align} \label{equation, decomposition form of p in K and L for abelian extensions}
\mathfrak{p} \mathcal{O}_K&=(\mathfrak{P}_1 \mathfrak{P}_2 \dots \mathfrak{P}_g)^{e_{\mathfrak{p}(K/F)}} =\left( \Pi_{\mathfrak{p}^{f_K}}(K/F)\right)^{e_{\mathfrak{p}(K/F)}} , \nonumber \\  
\mathfrak{p} \mathcal{O}_L&=  (\mathfrak{C}_1 \mathfrak{C}_2 \dots \mathfrak{C}_t)^{e_{\mathfrak{p}(L/F)}} =\left( \Pi_{\mathfrak{p}^{f_L}}(L/F)\right)^{e_{\mathfrak{p}(L/F)}}.
\end{align}
Since $L/F$ is Ostrowski,  $\Pi_{\mathfrak{p}^{f_L}}(L/F) \in \epsilon_{L/F}(\Cl(F))$. Equivalently $ \Pi_{\mathfrak{p}^{f_L}}(L/F)=\epsilon_{L/F}([\mathfrak{a}])$ for some $[\mathfrak{a}] \in \Cl(F)$.
Hence
	{\small
	\begin{equation} \label{equation, 1}
\mathcal{N}_{L/K}\left(\Pi_{\mathfrak{p}^{f_L}}(L/F)\right)=\mathcal{N}_{L/K}\left(\epsilon_{L/F}([\mathfrak{a}]) \right)=\mathcal{N}_{L/K}\left( \epsilon_{L/K} (\epsilon_{K/F}([\mathfrak{a}]))\right)=\left(\epsilon_{K/F}([\mathfrak{a}]) \right)^{[L:K]}.
	\end{equation}}
On the other hand,
\begin{equation} \label{equation, 2}
\mathcal{N}_{L/K}\left(\mathfrak{p} \mathcal{O}_L \right) =\left(\mathfrak{p} \mathcal{O}_K \right)^{[L:K]}=\left( \Pi_{\mathfrak{p}^{f_K}}(K/F)\right)^{[L:K].e_{\mathfrak{p}(K/F)}}.
\end{equation}	
 By relations \eqref{equation, decomposition form of p in K and L for abelian extensions}, \eqref{equation, 1} and \eqref{equation, 2} we get
 	\begin{equation}
 \Pi_{\mathfrak{p}^{f_K}}(K/F) =\left(\epsilon_{K/F}([\mathfrak{a}]) \right)^{\frac{e_{\mathfrak{p}(L/F)}}{e_{\mathfrak{p}(K/F)}}} \in \epsilon_{K/F}(\Cl(F)).
 	\end{equation}
Finally, since $K/F$ is a Galois extension, for each prime $\mathfrak{p}^{\prime}\neq \mathfrak{p}$	of $F$ we have $\Pi_{{\mathfrak{p}^{\prime}}^{f_K}}(K/F) \in \epsilon_{K/F}(\Cl(F))$, see part $(i)$ of Lemma \ref{lemma, basic properties of pre-Ostrowski group} below. Therefore $\Ost(K/F)=\{0\}$ as claimed.
	\end{proof}

\begin{remark}
	For $F=\mathbb{Q}$, triviality of $\Ost(L/F)$ is a classical result for \textit{cyclotomic} fields \cite[Proposition 2.6]{Zantema}.
\end{remark}

Leriche  proved that the Hilbert class field $H(K)$ of $K$ is a P\'olya field \cite[Corollary 3.2]{Leriche 2014}, which recently has been generalized to the triviality of $\Po(H(K)/K)$, see \cite[Corollary 2.9]{MR2}.   
Using the same method as in \cite[$\S$2]{MR2} we have:
\begin{theorem} \label{theorem, Ostrowski quotient for Hilbert class field is trivial}
	For a finite extension $L/K$  of number fields, the extension $H(L)/K$ is Ostrowski.
\end{theorem}

\begin{proof}
Using the principal ideal theorem, we have
\begin{equation*}
\epsilon_{H(L)/K}(\Cl(K)) =\epsilon_{H(L)/L} \left(\epsilon_{L/K}(\Cl(K)) \right) \subseteq  \epsilon_{H(L)/L} \left(\Cl(L) \right) =\{0\}.
\end{equation*}
Therefore 
\begin{equation*}
\Ost(H(L)/K) =\Po(H(L)/K) \subseteq \Po(H(L)/L) =\{0\}.
\end{equation*}
Indeed, the first equality follows from Definition \ref{definition, Ostrowski group}, the middle containment follows from part $(ii)$ of Lemma \ref{lemma, for P sub N sub M, if M/N is Galois are Galois, then Po(M/P) is contained in Po(M/N)} and the last equality follows from part $(iii)$ of Corollary \ref{corollary, some consequence of the main exact sequence for RPG}.

\end{proof}

\begin{remark}
There is another proof for Theorem \ref{theorem, Ostrowski quotient for Hilbert class field is trivial}: One can relativize  Leriche's method in \cite[proof of Proposition 3.1]{Leriche 2014} to show that all \textit{relative Ostrowski ideals} $\Pi_{\mathfrak{p}^f}(H(L)/K)$ belong to $\epsilon_{H(L)/K} \left(\Cl(K)\right)$. Further, in the case that $L/K$ is a Galois extension, using equality \eqref{equation, order of the Ostrowski quotient} one can see that
 the result of Brumer-Rosen in \cite[Proposition 2.4]{Brumer-Rosen} coincides with Theorem \ref{theorem, Ostrowski quotient for Hilbert class field is trivial}.	
\end{remark}

\begin{definition} \cite[Definition 2.3]{Leriche 2014} \label{definition, Polya extension}
An extension $L/K$ is said to be a \textit{P\'olya extension} if $\epsilon_{L/K}(\Po(K))=\{0\}$.
\end{definition}

\begin{remark}
There is another notion of P\'olya extension due to Spickermann \cite{Spickermann}: a finite Galois extension $L/K$ of number fields with Galois group $G$ is said to be a \textit{P\'olya extension}, in the sense of Spickermann, if $I(L)^G \subseteq I(K).P(L)$. Indeed, this definition is equivalent to $\Po(L/K)=\epsilon_{L/K}(\Cl(K))$, see \cite[page 11]{ChabertI}. In other words, for Galois extensions the notion of P\'olya extension in Spickermann's sense coincides with $\Ost(L/K)=\{0\}$.
\end{remark}

\noindent \textbf{Convention.} Throughout this article, by a \textit{P\'olya extension} we mean the same notion as in Definition \ref{definition, Polya extension}, and not in the sense of Spickermann.

The following  relation between P\'olya fields and P\'olya extensions has been found by Leriche:

\begin{proposition} \cite[Proposition 3.4]{Leriche 2014} \label{proposition, Leriche result on relation between Polya fields and Polya extensions}
Let $L/K$ be a P\'olya extension of Galois number fields. If all finite places are unramified in the estension $L/K$, then $L$ is a P\'olya field.
\end{proposition}

We aim to generalize the above result of Leriche. First we relativize the notion of P\'olya extension:

\begin{definition} \label{definition, relative Polya extension}
For $F \subseteq K \subseteq L$  a tower of finite extensions of number fields, the extension $L/K$ is called an \textit{$F$-relative P\'olya extension} whenever $\epsilon_{L/K}(\Po(K/F))=\{0\}$. Note that $\mathbb{Q}$-relative P\'olya extension is the same notion as in Definition \ref{definition, Polya extension}. 
\end{definition}

\begin{theorem} \label{theorem, generalization of Prop. 3.4. Leriche 2014}
Let $F \subseteq K \subseteq L$ be a tower of finite extensions of number fields. 
If both $K/F$ and $L/F$ are Galois extensions, then the following sequence is exact:
\begin{equation*} 
\{0 \} \rightarrow \Ker(\psi) \rightarrow \Coker(\gamma) \rightarrow \bigoplus _{\mathfrak{P}} \frac{\mathbb{Z}}{e_{\mathfrak{P} (L/K)} \mathbb{Z}}   \rightarrow  \frac{\Ost(L/F)}{\psi(\Ost(K/F))} \rightarrow \{0 \},
\end{equation*}
 where the maps $\gamma$ and $\psi$ are defined as follows ($\theta_{K/F}$ and $\theta_{L/F}$ are as in exact sequence \eqref{equation, main exact sequence for the Ostrowski quotient} and ``$\Inf$'' denotes the inflation map):
\begin{align*}
	\gamma:\frac{H^1(Gal(K/F),U_K)}{\theta_{K/F}\left(\Ker(\epsilon_{K/F})\right)} & \rightarrow \frac{H^1(Gal(L/F),U_L)}{\theta_{L/F}\left(\Ker(\epsilon_{L/F})\right)} \\
	[\sigma] \, \left(  \mathrm{mod}\, \, \theta_{K/F}\left(\Ker(\epsilon_{K/F})\right) \right)& \mapsto \Inf([\sigma]) \left(  \mathrm{mod}\, \, \theta_{L/F}\left(\Ker(\epsilon_{L/F})\right) \right)
\end{align*} 
and
 \begin{align*}
 \psi: \Ost(K/F) &\rightarrow \Ost(L/F) \\
 [\mathfrak{a}]  \left( \mathrm{mod}\, \, \epsilon_{K/F}(\Cl(F)) \right) & \mapsto \epsilon_{L/K}([\mathfrak{a}])  \left( \mathrm{mod}\, \, \epsilon_{L/F}(\Cl(F)) \right)
 \end{align*} 
\end{theorem}

\begin{proof}
Using exact sequence \eqref{equation, main exact sequence for the Ostrowski quotient} we find the following commutative diagram with exact rows of abelian groups:
{\footnotesize
	\begin{equation} \label{equation, commutative diagram for Galois extensions}
	\xymatrix{
		\{0 \} \ar[r] & \Ker(\epsilon_{K/F})   \ar[r]^{\theta_{K/F}} \ar[d]^{\text{inclusion}}  & H^1(Gal(K/F),U_K) \ar[r] \ar[d]^{\Inf} &  \bigoplus_{\mathfrak{P}}  \frac{\mathbb{Z}}{e_{\mathfrak{P} (K/F)} \mathbb{Z}} \ar[r] \ar[d]^{\text{inclusion}} & \Ost(K/F) \ar[d]^{\psi} \ar[r] & \{0\} \\
		\{0\} \ar[r] &\Ker(\epsilon_{L/F})  \ar[r]^{\theta_{L/F}}  & H^1(Gal(L/F),U_L)  \ar[r] & \bigoplus_{\mathfrak{P}}  \frac{\mathbb{Z}}{e_{\mathfrak{P} (L/F)} \mathbb{Z}} \ar[r] & \Ost(L/F) \ar[r] & \{0\} 
		 }
	\end{equation}}
Note that  $\psi$ is well-defined:

For $[\mathfrak{a}_1], [\mathfrak{a}_2] \in \Po(K/F)$, if $[\mathfrak{a}_1]=[\mathfrak{a}_2] \left( \mathrm{mod}\, \, \epsilon_{K/F}(\Cl(F)) \right) $, then $[\mathfrak{a}_1].[\mathfrak{a}_2]^{-1}=\epsilon_{K/F}([\mathfrak{b}])$ for some $[\mathfrak{b}] \in \Cl(F)$. Hence
\begin{equation*}
	\epsilon_{L/K}([\mathfrak{a}_1].[\mathfrak{a}_2]^{-1})=\epsilon_{L/K}(\epsilon_{K/F}([\mathfrak{b}]))=\epsilon_{L/F}([\mathfrak{b}]) \in \epsilon_{L/F} (\Cl(F)).
\end{equation*}
 
Equivalently, the following  diagram is commutative:
{\footnotesize
	\begin{displaymath}
\xymatrix{
	\{0 \} \ar[r] & \frac{H^1(Gal(K/F),U_K)}{\theta_{K/F}\left(\Ker(\epsilon_{K/F})\right)}  \ar[r] \ar[d]^{\gamma}  &  \bigoplus_{\mathfrak{P}}  \frac{\mathbb{Z}}{e_{\mathfrak{P} (K/F)} \mathbb{Z}} \ar[r] \ar[d]^{\text{inclusion}} & \Ost(K/F) \ar[d]^{\psi} \ar[r] & \{0\} \\
	\{0\} \ar[r] &\frac{H^1(Gal(L/F),U_L)}{\theta_{L/F}\left(\Ker(\epsilon_{L/F})\right)}  \ar[r]  & \bigoplus_{\mathfrak{P}}  \frac{\mathbb{Z}}{e_{\mathfrak{P} (L/F)} \mathbb{Z}} \ar[r] & \Ost(L/F) \ar[r] & \{0\} 
}
\end{displaymath}}

We show $\gamma$ is also well-defined:

For $[\sigma_1], [\sigma_2] \in H^1(Gal(K/F),U_K)$, if $[\sigma_1] =[\sigma_2] \left( \mathrm{mod}\, \theta_{K/F}\left(\Ker(\epsilon_{K/F})\right) \right)$, then $[\sigma_1]. [\sigma_2]^{-1} =\theta_{K/F}\left([\mathfrak{a}] \right)$ for some $[\mathfrak{a}] \in \Ker(\epsilon_{K/F})$. Using diagram \eqref{equation, commutative diagram for Galois extensions} we have
\begin{equation*}
\Inf([\sigma_1]. [\sigma_2]^{-1}) =\Inf(\theta_{K/F}\left([\mathfrak{a}] \right))=\theta_{L/F}\left([\mathfrak{a}] \right) \in \theta_{L/F}\left(\Ker(\epsilon_{L/F})) \right).
\end{equation*}
(Note that $\Ker(\epsilon_{K/F}) \subseteq \Ker(\epsilon_{L/F})$). Using the snake lemma, we obtain the desired exact sequence (Note that the snake lemma also yields $\gamma$ is one-to-one.).
\end{proof}

\begin{corollary} \label{corollary, some consequences of theorem generalizing Prop. 3.4. Leriche}
With the same notations and assumptions in Theorem  \ref{theorem, generalization of Prop. 3.4. Leriche 2014}, the following assertions hold:
\begin{itemize}
\item[(i)] If all finite places of $K$ are unramified in $L$, then $\Ost(L/F)=\psi(\Ost(K/F))$.
\item[(ii)] If $L/K$ is an $F$-relative P\'olya extension, then the following sequence is exact:
\begin{equation*} 
\{0 \} \rightarrow \Ost(K/F) \rightarrow \Coker(\gamma) \rightarrow \bigoplus _{\mathfrak{P}} \frac{\mathbb{Z}}{e_{\mathfrak{P} (L/K)} \mathbb{Z}}   \rightarrow  \Ost(L/F) \rightarrow \{0 \}.
\end{equation*}
\item[(iii)] If $\epsilon_{K/F}(\Cl(F))=\{0\}$, then the following sequence is exact: 
{\footnotesize
\begin{equation*} 
	\{0 \} \rightarrow \Ker(\epsilon_{L/K}\mid_{\Po(K/F)}) \rightarrow \Coker(\Inf) \rightarrow \bigoplus _{\mathfrak{P}} \frac{\mathbb{Z}}{e_{\mathfrak{P} (L/K)} \mathbb{Z}}   \rightarrow  \frac{\Po(L/F)}{\epsilon_{L/K}(\Po(K/F))} \rightarrow \{0 \}.
\end{equation*}}

\end{itemize}
\end{corollary}

\begin{proof}
Immediately follows from Theorem \ref{theorem, generalization of Prop. 3.4. Leriche 2014}. For proving part $(ii)$ note that since $L/K$ is an $F$-relative P\'olya extension, $\psi$ is the zero map.
	\end{proof}

\begin{remark}
Combining the first two parts or using part $(iii)$ of Corollary \ref{corollary, some consequences of theorem generalizing Prop. 3.4. Leriche}, we find a generalization of Leriche's result in Proposition \ref{proposition, Leriche result on relation between Polya fields and Polya extensions}.
\end{remark}

We recall that for a finite extension  $K/F$ of number fields, the \textit{relative genus field} of $K$ over $F$, denoted by $\Gamma(K/F)$, is the maximal abelian unramified extension of $K$ of the form $KF_{*}$ for some abelian extension $F_{*}/F$ \cite{Cornell}. By the \textit{genus field} of $K$,  denoted by $\Gamma(K)$, we mean the absolute genus field of $K$ (i.e. $F=\mathbb{Q}$). As an application of Proposition \ref{proposition, Leriche result on relation between Polya fields and Polya extensions}, Leriche proved that:
\begin{proposition} \cite[Proposition 3.4, Theorem 3.8]{Leriche 2014} \label{proposition, Leriche result about genus field}
If $K$ is an abelian number field then the genus field $\Gamma(K)$ of $K$ is P\'olya.
\end{proposition}

Unlike the Hilbert class field, the relative P\'olya group of the genus field, even for abelian number fields, is not necessarily trivial. For instance, one can show that for $K=\mathbb{Q}(\sqrt{-23})$, we have $\Po(\Gamma(K)/K)=\Cl(K)\simeq \mathbb{Z}/3\mathbb{Z}$,
see \cite[Example 2.14]{MR2}. Whereas using the following  principal ideal theorem of Terada \cite{Terada}, one can generalize Proposition \ref{proposition, Leriche result about genus field} for Ostrowski quotients in ``cyclic extensions'':

\begin{proposition} \cite{Terada} \label{proposition, Trada's principal ideal theorem}
Let $K/F$ be a finite cyclic extension of number fields. Then every ambiguous ideal class in $K/F$ will be principal in $\Gamma(K/F)$.
\end{proposition}

\begin{theorem} \label{theorem, Ostrowski quotient of genus field for cyclic extensions is trivial}
For $K/F$ a finite cyclic extension of number fields, the extension $\Gamma(K/F)/F$ is Ostrowski. 
\end{theorem}

\begin{proof}
By Proposition \ref{proposition, Trada's principal ideal theorem}, $\epsilon_{\Gamma(K/F)/K}(\Po(K/F))=\{0\}$, i.e. $\Gamma(K/F)/K$ is an $F$-relative P\'olya extension. Since all finite places of $K$ are unramified in $\Gamma(K/F)$, part $(ii)$ of Corollary \ref{corollary, some consequences of theorem generalizing Prop. 3.4. Leriche} proves that $\Ost(\Gamma(K/F)/F)=\{0\}$.
\end{proof}

\section{Relativization of the of pre-P\'olya group}
As stated before, for a Galois number field $L$ all the Ostrowski ideals above unramified primes in $L$ are automatically principal, whereas in the non-Galois case this is not true, in general. Based on this, Zantema introduced the notion of the \textit{pre-P\'olya field} and gave an interesting group theoretical description of the pre-P\'olya condition for number fields \cite[$\S$ 6]{Zantema}. Recently, the pre-P\'olya condition  has been generalized to the notion of the \textit{pre-P\'olya group} $\Po(-)_{nr}$ by Chabert and Halberstadt \cite{Chabert II}. 
Relativizing these notions, one can define:
\begin{definition} \label{definition, relative pre-Polya-Ostrowski group}
	The  \textit{pre-Ostrowski quotient} $\Ost(L/K)_{nr}$ for a finite extension $L/K$ of number fields, is the subgroup  of  $\Ost(L/K)$ defined as follows:
	\begin{equation} \label{equation, relative pre-Ostrowski group}
	\Ost(L/K)_{nr}=\frac{Po(L/K)_{nr}}{Po(L/K)_{nr} \cap\epsilon_{L/K}(Cl(K))},
	\end{equation}
	where 
	\begin{equation} \label{equation, relative pre-Polya group}
	Po(L/K)_{nr}=\left< \left[ \Pi_{\mathfrak{P}^f}(L/K)\right] \, :  \, \mathfrak{P} \, \text{is a prime in} \, K \, \text{unramified in} \, L \, , \, f \in \mathbb{N} \right>.
	\end{equation}
	In particular $\Ost(L/\mathbb{Q})_{nr}$ coincides with the ``\textit{pre-P\'olya group}'' $Po(L)_{nr}$ defined in \cite[$\S 1$]{Chabert II}. 
\end{definition}

\begin{lemma} \label{lemma, basic properties of pre-Ostrowski group}
For $L/K$  a finite extension of number fields,
\begin{itemize}
\item[(i)]
$\Ost(L/K)_{nr}$ is trivial if and only if $Po(L/K)_{nr} \subseteq \epsilon_{L/K}(\Cl(K))$. In particular if $L/K$ is a Galois extension, then $\Ost(L/K)_{nr} =\{0\}$.
\item[(ii)]
if $\epsilon_{L/K}(\Cl(K))=\{0\}$, then $\Ost(L/K)_{nr}=Po(L/K)_{nr}$;
\item[(iii)]
if $L/K$ is Galois and  all finite places of $K$ are unramified in $L$, then  
 $\Ost(L/K)_{nr}$ is trivial and
\begin{equation*}
\Po(L/K)_{nr}=\Po(L/K)=\epsilon_{L/K}(\Cl(K)).
\end{equation*}
\end{itemize}
\end{lemma}

\begin{proof}
	$(i)$. The first assertion is obvious by Definition \ref{definition, relative pre-Polya-Ostrowski group}. If $L/K$ is Galois, then for a prime $\mathfrak{p}$ of $K$ unramified in $L$,  all primes of $L$ above $\mathfrak{p}$ have the same residue class degree, say $f_{\mathfrak{p}}$, which implies that
	\begin{equation*}
\Pi_{\mathfrak{p}^{f_{\mathfrak{p}}}}(L/K)=\mathfrak{p} \mathcal{O}_L \in \epsilon_{L/K}(\Cl(K)).
	\end{equation*}
(Note that for the special case $K=\mathbb{Q}$ every Galois number field $L$ is a ``pre-P\'olya'' field, see \cite[$\S$ 6]{Zantema}).

$(ii)$ Immediately follows from Definition \ref{definition, relative pre-Polya-Ostrowski group}. 

$(iii)$ Since $L/K$ is  Galois and unramified, exact sequence \eqref{equation, main exact sequence for the Ostrowski quotient} yields $\Ost(L/K)$ is trivial or equivalently 
\begin{equation*}
\Po(L/K)=\epsilon_{L/K}(\Cl(K)).
\end{equation*}
The equality $\Po(L/K)_{nr}=\Po(L/K)$ holds since $L/K$ is unramified.
	\end{proof}

\begin{example} \label{example, Q(-4027)}
Let $K=\mathbb{Q}(\sqrt{-4027})$ and $\alpha$ be a root of $f(x)=x^3+10x+1$. One can show that $L=K(\alpha)$
is a cyclic unramified extension of $K$, see \cite[Example 2.14]{MR1}. By Lemma \ref{lemma, basic properties of pre-Ostrowski group}, $Ost(L/K)_{nr}$ is trivial and
\begin{equation*}
\Po(L/K)_{nr}=\Po(L/K)=\epsilon_{L/K}(\Cl(K)).
\end{equation*}

Indeed, since $L/K$ is unramified, by exact sequence \eqref{equation, main exact sequence for the Ostrowski quotient} we find
\begin{equation*}
\#\Ker(\epsilon_{L/K})=\#H^1(Gal(L/K),U_L)=\frac{\# \hat{H}^0(Gal(L/K), U_L)}{Q(Gal(L/K),U_L)},
\end{equation*}
where  $Q(Gal(L/K),U_L)$ denotes the the Herbrand quotient  of $U_L$ (as $Gal(L/K)$-module). One can easily show that 
\begin{equation*}
Q(Gal(L/K),U_L)=1/3, \quad \# \hat{H}^0(Gal(L/K), U_L)=1,
\end{equation*}
since $U_K=\{\pm 1\}$ and $N_{L/K}(-1)=-1$ (for more details see proof of Theorem \ref{theorem, generalization of Hilbert's theorem 94}). Therefore the capitulation kernel of  $L/K$ has order three. On the other hand, since $h_K=9$,  we have
\begin{equation*}
\#\Po(L/K)_{nr}=\#\Po(L/K)=\#\epsilon_{L/K}(\Cl(K))=3.
\end{equation*}
Hence $\Po(L/K)$ is a nontrivial proper subgroup of $\Cl(L)$, since $h_L=108$.
\end{example}

\begin{theorem} \label{theorem, for Galois extensions pre-Ostrowski group is generated by ramified primes}
	For a finite Galois extension $L/K$ of number fields,  we have
\begin{equation*}
\Ost(L/K)  \subseteq \frac{\Po(L/K)}{\Po(L/K)_{nr}}.
\end{equation*}
\end{theorem}

\begin{proof}
	Since $L/K$ is Galois, the containments	
	\begin{equation}
	Po(L/K)_{nr} \subseteq \epsilon_{L/K}(\Cl(K)) \subseteq \Po(L/K)
	\end{equation}
	prove the assertions.
\end{proof}

\section*{Acknowledgment}
The authors would like to thank the anonymous referees for their valuable comments and carefully reading the first draft.

\bibliographystyle{amsplain}

\end{document}